\documentclass[11pt]{article}
\usepackage[left=1in, right=1in, top=1in]{geometry}
\usepackage{amsmath}
\usepackage{amsfonts}
\usepackage{amssymb}
\usepackage{amsthm}
\usepackage{xcolor}
\usepackage{hyperref}
\usepackage{tikz-cd}
\usepackage{mathtools} 
\usepackage[normalem]{ulem} 
\usepackage{comment} 
\newtheorem{theorem}{Theorem}

\newtheorem{example}{Example}
\newtheorem{lemma}[theorem]{Lemma}

\usepackage{enumitem}

\newcommand{\compplain}{\frak{M}_{g,0,2r}}

\newcommand{\dpi}{dp \vert_{\pi}}

\newtheorem{corollary}[theorem]{Corollary}
\newcommand{\bb}{\mathbb{C}}
\newcommand{\om}{\mathcal{O}_M}

\newcommand{\oxp}{\mc{O}_{X,+}}

\newcommand{\z}{\mathbb{Z}}
\newcommand{\comp}{\frak{M}_{g,n,2r}}
 
\newcommand{\mcx}{\mc{X}}
 
\newcommand{\derscf}{\underline{\on{Der}}^{scf}_k} 

\newcommand{\compbranch}{\tilde{\frak{M}}_{d, \rho}}

\newcommand{\defpi}{\on{Def}_{\pi}}
\newcommand{\ox}{\mc{O}_X}
\newcommand{\oxprime}{\mc{O}_{X'}}
\newcommand{\nspunc}{\mc{P}_{NS}}
\newcommand{\rpunc}{\mc{P}_{R}}

\newcommand{\pp}{\mathbb{P}}
\newcommand{\aff}{\mathbb{A}}

\newcommand{\Spec}{\on{Spec} }

\newcommand{\mc}[1]{\mathcal{#1}}
\newcommand{\on}[1]{\operatorname{#1}}

\newcommand{\pull}{\pi^*} 
\newcommand{\ov}[1]{\overline{#1}}

\newcommand{\tilc}{\tilde{C}}
\newcommand{\tilx}{\tilde{X}}
\newcommand{\tils}{\tilde{S}}
\newcommand{\tilsup}{\tilde{\mc{D}}}

\newcommand{\oc}{\mc{O}_C}
\newcommand{\oct}{\mc{O}_{\tilc}}
 
\newcommand{\modtwoone}{\frak{M}_{2,1}}

\newcommand{\lgr}{\overset{\sim}{\longrightarrow}}

\newcommand{\os}{\mc{O}_S}

\newcommand{\otpr}{\mc{O}_{T'}}

\newcommand{\wh}[1]{\widehat{#1}}

\newcommand{\frk}[1]{\mathfrak{m}_{#1}}

\newcommand{\oto}{\mc{O}_{T_0} }

\newcommand{\otr}{\mc{O}_{\terd}}

\newcommand{\ot}{\mc{O}_T}

\newcommand{\mxx}{\mc{X}}

\newcommand{\spe}{\on{Spec}}

\newcommand{\mm}{\mc{M}}
 
\newcommand{\terd}{T_{\on{red}}}

\newcommand{\sS}{\on{sAffSch}} 
\usepackage[
 backend=biber,
 style=numeric,
 sorting=ynt
 ]{biblatex}
\title{Supermoduli space with Ramond punctures is not projected}
\author{Ron Donagi, Nadia Ott}
\date{}

\begin{document}

\maketitle
Supergeometry is the study of spaces described by a $\z_2$-graded sheaf of functions, $ \ox = \oxp \oplus \mc{O}_{X,-}$  whose sections obey the rule of signs, 
\[ ab = (-1)^{|a||b|} ba. \] 
Every superspace $X$ determines an ordinary space $M$ which naturally embeds into $X$, and one can ask if this embedding $M \subset X$ has a projection $X \to M$. If such a projection exists, then $X$  is said to be \emph{projected}. If, in addition, the projection $X \to M$ makes $X$ into a vector bundle over $M$, then $X$ is said to be  \emph{split}. One important class of examples of superspaces are the moduli spaces of super Riemann surfaces, known as \emph{supermoduli spaces} and denoted by $\frak{M}_{g,n_S,n_R}$. Here $g$ is the genus, $n_S$ is the number of Neveu-Schwarz punctures, and $n_R$ the (always even) number of Ramond punctures. In physics, integrals over supermoduli space calculate important quantities like the partition function and scattering amplitude. So far, the most successful computations of these quantities have relied upon the splitness of the associated supermoduli space; notably,  D'Hoker and Phong's computation of the scattering amplitude in genus $g=2$ \cite{d2002two}. 
However, most supermoduli spaces are not split or projected, \cite{donagi2015supermoduli}. Specifically, the cited work demonstrates that $\frak{M}_{g,n_S,0}$ is not projected for all $g \ge 5$ when $n_S=0$, and for all $g \ge n_S+1 \ge 2$ when $n_S \ge 1$.  This paper's main result is that most supermoduli spaces with Ramond punctures $\frak{M}_{g,0,n_r}$  are not projected.

\begin{theorem} \label{main:intro} Let $r >0$. The supermoduli space $\frak{M}_{g,0,2r}$ is not projected for all $g \ge 5r +1 \ge 6$.

\end{theorem} 

We start with the finite covering spaces of supermoduli space, \begin{equation} \label{cover}
p: \compbranch \to \frak{M}_{2,1,0}: (\pi: \tilx \to X) \to X
\end{equation}
parameterizing branched covers of genus $2$ super Riemann surfaces with a fixed degree $d$, a specified ramification configuration $\rho=(d_1, \dots, d_s)$, and a single branch point along the one NS puncture on $X$. Our main result about $\compbranch$ is that it is not projected (Lemma \ref{S1}).

Let us now consider those super Riemann surfaces $\tilx$ forming branched covers as in \eqref{cover}. 
The genus of $\tilx$ can be calculated using the usual Hurwitz formula,
\begin{equation} \label{formulahur}
g=1 + d + \frac{1}{2} \sum_{i=1}^s(d_i – 1)=1+\frac{3}{2}d - \frac{1}{2}s
\end{equation}   
The ramification configuration $\rho=(d_1, \dots, d_s)$  keeps track of the number of sheets $d_i$ that come together at the ramification points of $\pi$.  In our application below an even number $2r$ of the  $s$ ramification points will turn into Ramond punctures while the remaining $s-2r$ points will turn into NS punctures.  Furthermore, from \cite{donagi2015supermoduli} we know that $\tilx$ has Ramond punctures corresponding to those ramification points of $\pi$ with even local degree. Note that by formula \eqref{formulahur} there are always an even number of ramification points of even local degree. 
 
Each $\compbranch$ comes with a natural immersion into supermoduli space, 
\begin{equation} \label{i}
    i: \compbranch \to \compplain: (\pi: \tilx \to X) \to \tilx.
\end{equation}
[where $g$ is determined by \eqref{formulahur} and $r$ is as described above. \footnote{ In general, $\tilx$ also has NS punctures, and $i$ is the composition of $\compbranch \to \comp$ with the forgetful functor   $\comp \to \compplain$. }. In Lemma \eqref{S2} we use this immersion to deduce our main result about $\frak{M}_{g,0,2r}$, namely that it is not projected.



The last step of Theorem \ref{main:intro} is  the bound $g \ge 5r+1 \ge 6$. Fix $r >0$. For the sake of brevity, we will say that the supermoduli space $\compplain$ is \emph{realizable} if $g$ is a solution to \eqref{formulahur} for some $d$ and $\rho$. In other words, a supermoduli space $\compplain$ is realizable if it contains $i(\compbranch)$ for some $d, \rho$. From the above discussion, we know that a realizable supermoduli space is  not projected. 

\begin{lemma}\label{range:intro} $\compplain$ is realizable if and only if $g \ge 5r+1 \ge 6$
    
\end{lemma}

\begin{proof}

We consider a $d$-sheeted cover $\tilx \to X$ of a genus $2$ curve with a single branch point over which the fiber is specified by the ramification pattern $\rho = (d_1, … , d_s)$ with $\sum_i d_i = d$. The expression in (2) has to be a non-negative integer, namely the genus of $\tilx$. Conversely, Theorem 4 of \cite{husemoller1962ramified} assures us that there are no further constraints: for any $d$, $\rho$ such that the expression in (2) is a non-negative integer, there is a cover $\tilx \to X$ with the specified behavior.
 
As long as $g \ge 5r+1$ we can therefore construct a $d=g-(1+r)$-sheeted cover of $X \in \frak{M}_{2,1}$, branched over the single puncture in $X$ with $d_i = 2$ for $1 \le i \le 2r$ and $d_i=1$ for $i > 2r$. This gives a curve $\tilx$ of genus $g$ with $2r$ Ramond punctures and $g-(5r+1)$ Neveu-Schwarz punctures. (In the minimal case, $g=5r+1$, $d=4r$, and all the $d_i= 2$, so we have $2r$ Ramond punctures. When $g \ge 5r+1$, we still have $2r$ Ramond punctures but now also a positive number of NS punctures.) 

\end{proof}

The paper is structured as follows:  In Section \ref{split}, we look at various examples of split and projected supermanifolds. We also state two Corollaries, \ref{twoeight}, \ref{twoeleven} from \cite{donagi2015supermoduli} that allow one to construct a large class of non-projected supermanifolds via submanifolds and branched covers. 
 Secondly, in Section \ref{prelim}, we provide definitions for supercurves and super Riemann surfaces, and describe the two types of punctures one sees on a super Riemann surface, namely NS punctures and Ramond punctures. Furthermore, we give an account of divisors on supercurves, their connection to points, and clarify why Ramond punctures are not considered punctures. 
In Section \ref{branched}, we consider branched covers $\pi: \tilx \to X$ of super Riemann surfaces and describe the relationship between ramification and the types of punctures on $\tilx$. Lastly, in Section \ref{modbranch}, we define the morphisms in equations \eqref{cover} and \eqref{i}, and provide proof for Lemmas \ref{S1} and \ref{S2}. Additionally, the stackyness of $\compbranch$ is addressed in the appendix of this paper.
\newline

\textbf{Acknowledgements.} 
During the preparation of this work, Ron Donagi was supported in part by NSF grants DMS 2001673 and FRG 2244978, and by Simons HMS Collaboration grant $\#$390287. Nadia Ott was supported in part by Simons HMS Collaboration grant $\#$390287. We thank Eric D’Hoker whose question started us thinking about the material presented here.

\section{Preliminaries: Split and projected supermanifolds} \label{split} 


Let $X=(|X|, \mc{O}_X)$ be a supermanifold and let $J \subset \ox$  denote the ideal sheaf generated by $\mc{O}_{X}^-$.
Any supermanifold $X$ determines an ordinary manifold $X_{bos}:=(|X|, \mc{O}_X/J)$ called its \emph{bosonic reduction}. 
Henceforth, we set $M:=X_{bos}$.

 Given an ordinary manifold $M$ and a locally free sheaf  $\mc{X}$ on $M$, we can always construct a supermanifold, denoted by $S(M,  \mc{X})$, by applying the parity reversing-functor $\Pi$ to $\mc{X}$ and defining $S(M, \Pi \mc{X}):=\Spec_M \on{Sym}^{\bullet} \Pi \mc{X}$.  
\newline

\begin{example} \normalfont  The supermanifold  $S(M, J/J^2)$ determined by the pair $(M, J/J^2)$ is the normal bundle, $N_{M/X}$, to the embedding $M \subset X$. To see this, note that $J/J^2$ is the conormal sheaf associated to the embedding $M \subset X$. 
\end{example}

We say that $X$ is \emph{split} if there exists an isomorphism  $\phi: X \cong N_{M/X}$ of supermanifolds such that the restriction $\phi \vert_M$ is equal to the identity on $M$. A weaker condition than being split is being projected: The supermanifold $X$ is said to be \emph{projected} if the embedding $M \subset X$ has a projection $X \to M$. 
In terms of coordinates and transition functions, a supermanifold is projected if its even coordinates modulo nilpotents transform like coordinates on $M$, and 
it is split if, in addition, its odd coordinates transform like sections of the normal bundle $N_{M/X}$.

\begin{example} \label{splitex} \normalfont All affine superspaces $\mathbb{A}^{m|n}$ and all projective superspaces $\mathbb{P}^{m|n}$ are split. Specifically, \begin{align*}
    \aff^{m|n} & = \on{Spec}_{\aff^m} \on{Sym}^{\bullet} \Pi \mc{O}_{\aff^m}^{\oplus n} \\ 
    \pp^{m|n} & = \on{Spec}_{\pp^m} \on{Sym}^{\bullet}  \Pi \mc{O}(1)^{\oplus n}.
\end{align*} 
 
\end{example}

\begin{example}\normalfont 
    \label{oneonesplit} Every supermanifold of dimension $(m|1)$ is split.  
More generally,  the \emph{first truncation} $X^1:=(|X|, \ox/J^2)$ of any supermanifold $X$ is  split. 
\end{example}

In general, a projected supermanifold need not be split. However, there are some special cases where the two notions are equivalent: 
\begin{example} \label{trunc} \normalfont A supermanifold of dimension $(m|2)$ is split if and only if it is projected. More generally, the second truncation $X^2:=(|X|, \ox/J^3)$ of any supermanifold $X$ is split if and only if projected. 
This is (roughly) because the transformation laws for the odd coordinates on $X^2$ remain unchanged from those on the split supermanifold $X^1$ when ``adding back" variables in $J^2/J^3$. 
\end{example} 

The truncations $X^1$ and $X^2$ are the first two terms in a finite sequence of truncations,  $X^i:=(|X|, \ox/J^{i+1})$ which approximate $X$ in the sense of the embeddings \[ M \subset X^1 \subset \cdots X^i \subset \cdots \subset X. \] 

A supermanifold $X$ is projected or split if and only if all of its truncations $X^i$ are projected or split. 
Since $X^1$ is always split, the first interesting case is $X^2$, the second truncation.  $X^2$ determines a class $ \omega_2 \in H^1(M, \mc{T}_M \otimes J^2/J^3)$ called the \emph{first obstruction to splitting}. Indeed, the surjection $\ox/J^3 \to \ox/J^2=\om \oplus J/J^3$ making $X^2$ a square-zero extension of $X^1$ is an isomorphism on the odd components, and the conormal sequence associated to $(\ox/J^3)^+ \to \om$, 
 \[ 
    0 \to J^2/J^3 \to \Omega_{X^2}^1 \vert_M \to \Omega_M^1 \to 0,  \] 
identifies $X^2$ with a class $\omega_2$ in $\on{Ext}^1(\Omega_M^1, J^2/J^3) \cong H^1(M, \mc{T}_M \otimes J^2/J^3)$ such that $\omega_2=0$ if and only if $0 \to J^2/J^3 \to (\ox/J^3)^+ \to \om \to 0$ is split exact. In this case,  $X^2$ is projected, and thus also split by example \ref{trunc}. 
\newline

 
It is not difficult to find examples of supermanifolds which are \emph{not} projected.

\begin{example}\label{superconic}  \normalfont The superconic defined by the equation $z_1^2 + z_2^2 + z_3^2 + \theta_1 \theta_1$ in $\on{Proj} \bb[z_1, z_2, z_3 | \theta_1, \theta_2]$ is not projected. 
    
\end{example} 

\begin{example}
    \label{firstorder} \normalfont  Every non-trivial \emph{odd} first order deformation of a super Riemann surface supermanifold is not projected, see \cite{donagi2015supermoduli}. 
\end{example}

The following two Corollaries  from \cite{donagi2015supermoduli} can be used to construct a large class of non-projected supermanifolds. 

\begin{corollary}[Corollary 2.8] \label{twoeight} Let $\pi: \tilx \to X$ be a finite cover of supermanifolds. If $\omega_2(X) \neq 0$, then $\omega_2(\tilx) \neq 0$. Furthermore, $X$  is split if and only if $\tilx$ is split. 

\end{corollary}

\begin{corollary}[Corollary 2.11] \label{twoeleven} Let $X$ be a supermanifold and suppose $X' \subset X$ is a sub-supermanifold with bosonic reduction $M' \subset M$ such that the normal sequence, 
\[ 0 \to TM \to TM \vert_{M'} \to N_{M'/M} \to 0 \]
is split exact. If $X'$ is not projected, then $X$ is not projected. 
\end{corollary}  

\noindent \textbf{Example.} Let us apply Corollary \ref{twoeleven}  to the superconic $X \subset \pp^{2|2}$ from  example \ref{superconic}. Since $X$ is not projected, the normal sequence associated to the embedding $X_{bos} \subset \pp^2$ is not split exact by Corollary \ref{twoeleven}.  
\newline

\section{Preliminaries: Supercurves, divisors on supercurves, super Riemann surfaces and their punctures } \label{prelim}

\subsection{Supercurves}

A \emph{supercurve} $S$ is a compact, connected supermanifold of dimension $(1|1)$. The bosonic reduction of a supercurve is an ordinary projective curve.

\begin{example} \normalfont
    The superprojective line $\mathbb{P}^{1|1}$ is a supercurve whose bosonic reduction is the ordinary projective line, $\mathbb{P}^1$.
\end{example}

A \emph{family of supercurves} over a supermanifold $T$ is a smooth, proper morphism of supermanifolds $\pi: S \to T$ of relative dimension $(1|1)$. The bosonic reduction of $\pi: S \to T$ is the family of ordinary curves, $\pi_{bos}: S_{bos} \to T_{bos}$.
\begin{example} \normalfont An example of a family of supercurves is  $\mathbb{P}^{1|1} \times T \to T$, where $T$ can be any supermanifold. Its bosonic reduction is $\mathbb{P}^1 \times T_{bos} \to T_{bos}$.
\end{example}

In general, we write $S$ for a family of supercurves and make no explicit mention of the base supermanifold.  We will sometimes use the phrase ``in the split case" to indicate that the base of a family is an ordinary scheme. If the base is $\Spec \bb $, then we will refer to $S$ as a ``single" supercurve. 
\newline

Any topological invariant assigned to a supermanifold is determined by the bosonic reduction of that supermanifold.  In particular, the genus of a supercurve is the genus of its bosonic reduction.


\subsection{Divisors on supercurves}

Our treatment of divisors on supermanifolds follows the style of Chapter 6 in \cite{hartshorne2013algebraic}. 
\newline 

Let $X$ be a supermanifold and let $\mc{K}$ denote the sheaf of total quotient superrings (see \cite{hartshorne2013algebraic} for the ordinary definition).
A \emph{Cartier divisor}  on  $X$ is a global section of $\mc{K}^*/\mc{O}_X^*$, where  $\mc{K}^* \subset \mc{K}$ is the subsheaf of invertible elements in $\mc{K}$. As in the ordinary case, a Cartier divisor on $X$ can be described by a collection $f_i \in \Gamma(U_i, \mc{K}^*)$ on an open cover $U_i$ of $X$ such that for all $i,j$, $f_i/f_j \in \Gamma(U_i \cap U_j, \mc{O}_X^*)$.  
\newline

\noindent \textbf{Effective and prime divisors} A  Cartier divisor is said to be \emph{effective} if $f_i \in \Gamma(U_i, \mc{O}_{U_i,+})$ for all $i$. Every effective Cartier divisor on $X$ has an associated closed sub-supermanifold $\mc{P} \subset X$ of codimension $1|0$ defined by the ideal sheaf $\mc{I}_{\mc{P}} \subset \mc{O}_X$ locally generated by the $f_i$. We say that $\mc{P} \subset X$ is a \emph{prime divisor}  if its bosonic reduction $P=\mc{P}_{bos}$ is a prime divisor  on $M=X_{bos}$. 

\begin{example} \normalfont Let $(z, \theta)$ be a choice of local coordinates on a supercurve $S$ over a base with one odd coordinate $\eta$. The divisor defined by the equation $z - z_0-\theta \eta=0$ is a prime divisor on $S$.
\end{example}

 \begin{lemma} \label{trivialven} On any supermanifold $X$, $\mc{O}_X^*=\mc{O}_{X,+}^*$ and $\mc{K}^* = \mc{K}_+^*$. In particular, a one-to-one correspondence exists between Cartier divisors on $X$ and Cartier divisors on $X_{ev}$. 

 \end{lemma}
 \begin{proof} A section $f \in \mc{O}_X$ 
  is invertible if and only if its image $\ov{f}$ in $\mc{O}_X/J$ is invertible. Therefore, if $f \in \mc{O}_X^*$, then $\ov{f} \neq 0$ and, by homogeneity, $f$ is even since $\ov{f} \neq 0$ is  even. 
 \end{proof}

\begin{corollary} \label{evensame}  Let $S$ be a family of supercurves over an ordinary scheme. Then there is a one-to-one correspondence between Cartier divisors on $S$ and Cartier divisors on the bosonic reduction $C$ of $S$. In particular, every Cartier divisor on $C$ extends uniquely to a Cartier divisor on $S$.  
    
\end{corollary} 

\begin{proof} Over an ordinary scheme, $S_{bos}=S_{ev}$ since $\mc{O}_{S,+}=\mc{O}_S/J$. Now apply Lemma \ref{trivialven}.
    
\end{proof}


\subsubsection{Points and divisors}

A \emph{marked point} on a family of supercurves $\pi: S \to T$ is the image of a section $s: T \to S$.  When $T= \Spec \bb$, marked points are the same as closed points. 
 \newline

\begin{example} \normalfont Consider the (trivial) family $\pi: \mathbb{A}^{1|1} \times \mathbb{A}^{0|1}  \to \mathbb{A}^{0|1}$, and choose coordinates
 \[ \mathbb{A}^{1|1} \cong \Spec \bb[z, \theta], \ \ \mathbb{A}^{0|1} \cong \Spec \bb[\eta] \]
 where the isomorphism on the right is for both copies of $\aff^{0|1}$. 
 A marked point $s: \mathbb{A}^{0|1} \to \mathbb{A}^{1|1} \times \mathbb{A}^{0|1} $ is a morphism of superrings
 \begin{equation*}
      \bb[\eta][z, \theta] \to \bb[\eta]:  
  \eta  \mapsto \eta, 
  z   \mapsto a, 
  \theta  \mapsto b \eta 
 \end{equation*}
 for some $a,b \in \bb$.

 \end{example}

\begin{lemma}
    \label{extended} Let $S$ be a family of supercurves over a purely bosonic scheme. Then there is a one-to-one correspondence between prime divisors on $S$ and marked points on $S$. 
\end{lemma}

\begin{proof}
    Suppose $T$ is purely bosonic. Then any marked point $s: T \to S$ on $S$ must factor through $C=S_{bos}$, and so we have a one-to-one correspondence between marked points on $C$ and marked points on $S$. On the other hand, marked points on $C$ are the same as prime divisors on $C$, and every prime divisor on $C$ extends uniquely to $S$ by Corollary \ref{evensame}. 
\end{proof}

However, there is \emph{no} duality between prime divisors and marked points on a supercurve over a general superscheme. In the next section we will show that the duality is restored,  regardless of the base superscheme, in the presence of a superconformal structure. 



    


    

\subsection{Super Riemann surfaces}
 A super Riemann surface $X=(S, \mc{D})$ is a supercurve $S$ equipped with a rank $(0|1)$ distribution $\mc{D} \subset \mc{T}_S$ which is maximally non-integrable in the sense of the following short exact sequence, 
\begin{equation} \label{definition}
    0 \to \mc{D} \to \mc{T}_{S/T} \to \mc{D}^{\otimes 2} \to 0.
\end{equation} 
An isomorphism of super Riemann surfaces $\phi: X \to X'$ is an isomorphism $\phi: S \to S'$ of supercurves such that $\phi^* \mc{D}'=\mc{D}$. 
\newline

The distribution $\mc{D}$ is called a superconformal structure, and there always exists a special choice of local coordinates $(z, \theta)$ on $S$ such that
 $\mc{D}$ is locally generated by the odd vector field, 
 \begin{equation} \label{localform} D = \frac{\partial}{\partial \theta} + \theta \frac{\partial}{\partial z}. \end{equation}
 \newline 

\noindent \emph{Duality.} In the last section we define a duality between points and divisors on a supercurve over a purely bosonic scheme, and noted that this duality does not hold over general superschemes. It turns out that on super Riemann surfaces the defined duality holds \emph{in general}, i.e., with no condition on the base superscheme. This is because the maximal non-integrability condition on $\mc{D}$ ensures that there exists a unique point on each  divisor tangent to $\mc{D}$. 

More precisely, let $p=(z_0, \theta_0)$ be a point on $S$, and let $v_{p}:= D \vert_{p}$ be the restriction of the local generator \eqref{localform} of $\mc{D}$ to the point $p$. The pair $(p, v_{p})$ determines a  prime divisor $P$ on $S$ which is locally defined by the equation $z-z_0 - \theta \theta_0=0$. By construction, $\mc{D}$ is tangent to $P$ at $p \in P$, and the maximal non-integrability condition on $\mc{D}$ ensures that $p$ is the unique point on $P$ tangent to $\mc{D}$. 
\newline 

\noindent \emph{Superconformal vector fields.}  Let $X=(S,\mc{D})$ be a super Riemann surface, and let $I$ be an $\mc{O}_S$-module. 
A superderivation $\delta: \os \to I$ is a $k$-linear map, satisfying the super Leibniz rule
\begin{equation}\delta(ab)=\delta(a)b + (-1)^{|\delta||a|} \ a  \delta(b), \ a,b \in I \end{equation} We can organize the set of derivations into a sheaf of super vector spaces on $S$, 
\[ \underline{\on{Der}}_k(\mc{O}_S,I)= \on{Der}_k(\mc{O}_S,I) \oplus \on{Der}_k(\mc{O}_S,\Pi I). \]
where $\on{Der}_k(\mc{O}_S,I)$ is the space of all grading-preserving derivations, and $\on{Der}_k(\mc{O}_S,\Pi I)$ is the space of grading-reversing derivations. The sheaf $\underline{\on{Der}}_k(\mc{O}_S,\os)$ is the tangent sheaf $\mc{T}_S$ on $S$, and $\mc{T}_S \otimes I=\underline{\on{Der}}_k(\mc{O}_S,I)$ where
\[ (\mc{T}_S \otimes I)^+= \on{Der}_k(\mc{O}_S,I), \ \text{and} \  (\mc{T}_S \otimes I)^-=\on{Der}_k(\mc{O}_S,\Pi I). \]

On the super Riemann surface $X$, we can consider the subsheaf  $\underline{\on{Der}}_k^{scf}(\mc{O}_X,\ox) \subset \underline{\on{Der}}_k(\mc{O}_X,\ox)=\mc{T}_X$ of derivations  preserving the superconformal structure $\mc{D}$.  A derivation $V \in \mc{T}_X$ is in $\underline{\on{Der}}_k^{scf}(\mc{O}_X,\ox)$ if and only if 
$ [V, \mc{D}] \in \mc{D}$. Henceforth we set $\mc{A}_X:=\underline{\on{Der}}_k^{scf}(\mc{O}_X,\ox)$. The maximal non-integrability condition on $\mc{D}$ implies that $V \in \mc{T}_X/\mc{D}$, and \begin{equation} \label{naked} \mc{A}_X \cong \mc{T}_X/\mc{D} \cong \mc{D}^2.  \end{equation}    If $I$ is an $\ox$-module, then $\derscf(\ox,I) \cong \mc{A}_X \otimes I$, where
\[ (\mc{A}_X \otimes I)^+= \on{Der}(\mc{O}_X,I), \ \text{and} \  (\mc{A}_X \otimes I)^-=\on{Der}(\mc{O}_X,\Pi I).\] 
Here we are using the fact that  $\mc{A}_X$ inherits the structure of a sheaf of $\ox$-modules from $\mc{D}^2$, and so the tensor product $\mc{A}_X \otimes_{\ox} -$ is well-defined.
\newline 

\begin{lemma} \label{technical} $I \cong \ox \oplus \Pi \ox$, then \begin{align*}
\on{Der}_k(\ox,I) & \cong \underline{\on{Der}}_k(\ox, \ox) \cong \mc{T}_X \\
\on{Der}_k^{scf}(\ox,I) & \cong \underline{\on{Der}}_k^{scf}(\ox, \ox) \cong \mc{A}_X
\end{align*}
\end{lemma}

\begin{proof} \begin{align*}
    \on{Der}_k(\ox,I) & \cong (\mc{T}_X \otimes_{\ox} (\ox \oplus \Pi \ox))_+=(\mc{T}_X \oplus \Pi \mc{T}_X)_+\\
    {} & =\mc{T}_X^+ \oplus (\Pi T_X)^+=T_X^+ \oplus \mc{T}_X^-\\
    {}& =\mc{T}_X
\end{align*}

\end{proof}

\noindent \emph{First order deformations.} A first order deformation of a super Riemann surface $X$ is a family of super Riemann surfaces  $f: X' \to \Spec D$ over the Spec of the super dual numbers $D=\bb[t, \eta]/(t^2, \eta t), |t|=0, |\eta|=1$ making the following diagram commute, 
\begin{equation} \label{deformSRS}
    \begin{tikzcd}
        X \arrow[r] \arrow[d, "f_0"] & X' \arrow[d,"f"] \\
        \Spec \bb \arrow[r] & \Spec D.
    \end{tikzcd}
\end{equation}
We say that a first order deformation $X'$ of $X$ is trivial if there exists an isomorphism $X' \cong X \times \Spec D$ restricting to the identity on $X$. 

\begin{lemma} \label{firstorderSRS}
The set of isomorphism classes of first order deformations  of $X$ is equal to $H^1(X,\mc{A}_X)$. \end{lemma}
\begin{proof}
The diagram \eqref{deformSRS} induces the following morphism  of short exact sequences  on structure sheaves, 
\begin{equation}
    \begin{tikzcd}
        0 \arrow[r]  & I \arrow[r, "{\cdot}t{,}{\cdot}\eta"]  & \mc{O}_{X'} \arrow[r] & \ox \arrow[r] &  0 \\ 
        0 \arrow[r]  & J=(t, \eta) \arrow[r] \arrow[u]  & D \arrow[r] \arrow[u,"f"] & \bb \arrow[r] \arrow[u, "f_0"] &  0
    \end{tikzcd}
\end{equation}
where $I:=\on{ker}(\oxprime \to \ox)$, $J:=\on{ker}(D \to \bb)$.
Any  first order deformation of is locally on $X$ isomorphic to the trivial deformation. Let $U_i$ be a trivializing cover for $X$, and let $\phi_i: X' \vert_i \lgr U_i \times \Spec D$ be the respective isomorphism. Then $\Phi_{ij}:=\phi_{ij}/\phi_{ji}$ is an automorphism of the trivial deformation $U_{ij} \times \Spec D$. Giving an automorphism of the trivial deformation is equivalent to giving an  even derivation $\delta: \mc{O}_{ij'} \to i_*I$ preserving the superconformal structure on $X$, so $\delta $ is a section of $\on{Der}^{scf}(i^*\oxprime, I)(U_{ij}) = \on{Der}^{scf}(\ox, I)(U_{ij})$.  Here $i$ denotes the top horizontal arrow $X \to X'$ in  \eqref{deformSRS}. Furthermore, since $I \cong f_0^*(J)= \ox \oplus \Pi \ox$, we have an isomorphism \[ \on{Der}^{scf}(\ox, I)  \cong \mc{A}_X, \] 
by Lemma \ref{technical},  and so $\delta \in \mc{A}_X(U_{ij})$.
The set of isomorphism classes of first order deformations of $X$ is therefore a torsor for $H^1(X, \mc{A}_X)$, and there is a canonical isomorphism from the set of isomorphism classes of first order deformations of $X$ to $H^1(X, \mc{A}_X)$ sending the trivial deformation of $X$ to $0 \in H^1(X, \mc{A}_X)$. 
\end{proof}

\subsubsection{Super Riemann surfaces with Neveu-Schwarz (NS) punctures } A \emph{Neveu-Schwarz (NS) puncture} on a super Riemann surface $X$ is just a choice of marked point, $s: T \to X$.  A super Riemann surface with $n_S$ NS punctures is the data $(\pi: X \to T, \mc{D}, s_1,\dots, s_n: T \to X)$.  An isomorphism of super Riemann surfaces with NS punctures is an isomorphism $\phi: X \to X'$ of super Riemann surfaces such that $\phi^{-1}(\mc{P}_{NS}')= \mc{P}_{NS}$. 
\newline

\noindent \emph{Duality.} The duality between points and divisors on super Riemann surfaces identifies each NS puncture with a prime divisor $\mc{P}_{NS} \subset X$ of degree $n_S$. 
\newline 

\noindent \emph{Superconformal vector fields.}  The sheaf of superconformal derivations on a super Riemann surface with NS punctures is the subsheaf of $\mc{A}_X$ whose sections vanish along $\nspunc$. Specifically, if $X=(S, \mc{D})$ is a super Riemann surface and $\nspunc$ is the NS divisor, then the sheaf of superconformal vector fields is the subsheaf $\mc{A}_X(-\mc{P}_{NS}) \subset \mc{A}_X$, where
\begin{equation} \label{NSder} \mc{A}_X(-\mc{P}_{NS})= \derscf(\ox, \ox) \otimes \mc{O}(-\mc{P}_{NS}) \cong \derscf(\ox, \mc{O}(-\mc{P}_{NS})) \end{equation} and 
\begin{equation} \label{confNS} \mc{A}_X(-\nspunc) \cong (T_X/\mc{D})(-\nspunc) \cong \mc{D}^2(-\nspunc).\end{equation}
\newline 

\noindent \emph{First order deformations.} A first order deformation of a super Riemann surface $X=(S, \mc{D})$ with NS punctures $\mc{P}_{NS}$ is a first order deformation $X'$ of $X$ and divisor $\mc{P}_{NS}'$ whose pull back to $X$ is equal to $\nspunc$. To prove the next lemma, replace  $\mc{A}_X$ in the proof of Lemma \ref{deformSRS} with $\mc{A}_X(-\nspunc)$. 

\begin{lemma} \label{deformSRSns} The set of isomorphism classes of first order deformations of $X=(S,\mc{D}, \nspunc)$ is equal to $H^1(X,\mc{A}_X(-\nspunc))$.
    
\end{lemma}

\subsubsection{Super Riemann surfaces with Ramond punctures.} Let $\pi: S \to T$ be a family of supercurves.  A \emph{Ramond divisor} on $S$ is prime divisor $\mc{P}_R \subset S$ for which there exists a rank $0|1$ distribution $\mc{D} \subset \mc{T}_{S/T}$ fitting into the short exact sequence
\[ 0 \to \mc{D} \to \mc{T}_{S/T} \to \mc{D}^2(\mc{P}_R) \to 0. \]
The prime components $\mc{P}_i \subset \mc{P}_R$ are called \emph{Ramond punctures}. A super Riemann surface with $n_R$ Ramond punctures is the data $(\pi: S \to T, \mc{P}_R, \mc{D})$ where $\mc{P}_R$ is a Ramond divisor of degree $n_R$. The number of Ramond punctures is always even, and so we set $n_R=2r$. 

We can always find local coordinates $(z, \theta)$ on $S$ such that $\mc{P}_R$ is locally defined by the equation $z=0$ and
 $\mc{D}$ is locally generated by the odd vector field, 
 \begin{equation} \label{localformramond} D = \frac{\partial}{\partial \theta} + z \theta \frac{\partial}{\partial z}. \end{equation}
 \newline 

\noindent \emph{Duality.}  The duality between points and divisors we saw on super Riemann surfaces, and super Riemann surfaces with NS punctures,  fails in the presence of Ramond punctures . Consequently, there is no canonical way to associate a set of marked points to the Ramond divisor $\rpunc$. This is in stark contrast to what we saw with NS punctures, which start out as marked points, but to which we can canonically associate a divisor $\nspunc$ using duality. Let us explain this in more detail.  First, note that $\mc{D}$ is non-integrable away from $\rpunc$ but \emph{integrable} along $\rpunc$ since \[ [ , \ , \ ]: \mc{D}^2 \lgr (\mc{T}_{S/T}/\mc{D})(-\rpunc)=0 \in \mc{D}, \]
 where $[ \ , \ ]$ is the supercommutator.  Recall from ordinary geometry that an integrable distribution of rank $k$ determines a submanifold of dimension $k$ which is tangent to the distribution at every point. Since $\mc{D}$ is integrable along $\rpunc$, it determines a submanifold of $\rpunc$ of dimension $0|1$. Since $\rpunc$ is a divisor, and thus of dimension $0|1$, this submanifold must be equal to $\rpunc$, and so that $\rpunc$ is  tangent to $\mc{D}$ at every point. In particular, there does \emph{not} exist a unique point on $\rpunc$ which we can use to define a duality.  In this sense,  \emph{Ramond punctures are not marked points}! 
\newline

\noindent \emph{Superconformal vector fields.} The  subsheaf  $\underline{\on{Der}}_k^{scf}(\mc{O}_X,\ox) \subset \underline{\on{Der}}_k(\mc{O}_X,\ox)=\mc{T}_X$ of derivations preserving the superconformal structure $\mc{D}$  \footnote{Any derivation that preserves $\mc{D}$ also preserves the Ramond divisor, $\rpunc$, since $\rpunc$  is part of the structure of $\mc{D}$.} on a super Riemann surface with Ramond punctures is defined by the following condition: Set $\mc{A}_X:=\underline{\on{Der}}_k^{scf}(\mc{O}_X,\ox)$. A derivation $V \in \mc{T}_X$ is a section of $\mc{A}_X$ if and only if 
$ [V, \mc{D}] \in \mc{D}$ if and only if $V \in (\mc{T}_X/\mc{D}))(-\rpunc)$, where the last condition follows from the  maximal non-integrability condition on $\mc{D}$. Thus, \begin{equation} \label{nakRam} \mc{A}_X \cong \mc{T}_X/\mc{D}(-\rpunc) \cong \mc{D}^2.  \end{equation}   
If $I$ is an $\ox$-module, then \begin{equation} \label{longlist} \derscf(\ox,I) \cong \mc{A}_X \otimes I, \end{equation} where
\[ (\mc{A}_X \otimes I)^+= \on{Der}^{scf}(\mc{O}_X,I), \ \text{and} \  (\mc{A}_X \otimes I)^-=\on{Der}^{scf}(\mc{O}_X,\Pi I).\] 
\newline

\noindent \emph{First order deformations. } A first order deformation of  $X=(S,\mc{P}_{R}, \mc{D})$ is a super Riemann surface with Ramond punctures $X'=(f: S' \to \Spec D, \mc{P}_R', \mc{D}')$ over $\Spec D$ making the diagram in \eqref{deformSRS} commute, and such that the restriction of $\rpunc'$ to $X$ is equal to $\rpunc$. 
\begin{lemma} \label{deformRamond}
    The set of isomorphism classes of first order deformations of $X=(S,\mc{P}_{R}, \mc{D})$ is equal to $H^1(X, \mc{A}_X)$.
\end{lemma}
\section{Branched Covers} \label{branched} 

In this section, we define branched covers of supercurves and super Riemann surfaces. 

\subsection{Branched covers of supercurves}

Let $S$ be a supercurve,  let $\mc{B} \subset S$ be a closed subscheme of codimension $(1|0)$, and let $U=S -\mc{B}$.  A \emph{branched cover} of $S$ with branch locus $\mc{B}$ is a  finite, surjective morphism $\pi: \tils \to S$ such that the local morphisms $\pi_p^{\#}: \mc{O}_{S,f(p)} \to \mc{O}_{\tils, p}$ are \'etale for all $p \in U$.  The cover $\pi$ may be ramified along a closed subscheme $\mc{R} \subset \pi^{-1}(\mc{B}) \subset \tils$ of codimension $1|0$ called the ramification divisor of $\pi$. 
The bosonic reduction $\pi_{bos}: \tilc \to C$ of $\pi: \tils \to S$ is ramified along a divisor $R=\mc{R}_{bos}$ with branch locus $B=\mc{B}_{bos}$.

An isomorphism of branched covers is a pair of isomorphisms of supercurves $\phi: S \lgr S'$ and $\tilde{\phi}: \tils \to \tils'$ such that the following diagram commutes, \begin{equation} \label{isomorphism} 
\begin{tikzcd}
     \tils \arrow[d, "\pi"] \arrow[r, " \tilde{\phi}"] & \tils' \arrow[d, "\pi'"]            \\
S \arrow[r, "\phi"] & S'. 
\end{tikzcd}
\end{equation} 
and such that $\phi(\mc{B})=\mc{B}'$ and $\tilde{\phi}(\mc{R}) = \mc{R}'$.
\newline

\noindent \emph{Local description.} Let $\pi: \tils \to S$ be a branched cover,  and let $\mc{P} \subset \mc{R}$ be a connected component of the ramification divisor of $\pi$. Near $\mc{P}$, we can find (local) coordinates $(w, \theta')$ and $(z, \theta) $ on  $\tils$ and $S$, respectively, such that $\mc{B} \subset S$  is defined by the equation $z=0$ and  $\pi$ is locally of the form
\begin{equation} \label{localdesc} z=w^k, \ \ \theta=w^{\ell} \theta' \end{equation}
for some $k > 0, \ell \ge 0$.
\newline

\noindent \emph{Branched covers over bosonic  schemes.}  Let $\pi: \tils \to S$ be a branched cover over a purely bosonic scheme. Then $S$ and $\tils$ are split, and there exist line bundles $L$ and $\tilde{L}$ on $C=S_{bos}$ and $\tilc=\tils_{bos}$, respectively, such that 
\[ \mc{O}_S= \mc{O}_C \oplus \Pi L, \ \ \mc{O}_{\tils}=\mc{O}_{\tilc} \oplus \Pi \tilde{L}, \]
and the even and odd components of $\pi^{\#}$, \[ \pi_0^{\#} = \pi_{bos}: \oc \to (\pi_0)_*\oct, \ \ \Pi \pi_1^{\#} : L \to (\pi_0)_*\tilde{L} \] are, respectively, a branched cover of the curve $C$, and a morphism of line bundles on $C$. 
From this description, we see that, in the split case, giving a branched cover of $S$ is equivalent to giving a branched cover  $\pi_0: \tilc \to C$ of $C$, a line bundle $\tilde{L}$ on $\tilc$, and a morphism $\Pi \pi_1^{\perp}: \pi_0^*L \to \tilde{L}$ of line bundles on $\tilc$. Note that the adjunction map sends $\Pi \pi_1^{\perp}$ to a morphism $\Pi \pi_1^{\#}:  L \to (\pi_0)_* \tilde{L}$ of $\oc$-modules. 
\newline 

\noindent \emph{Semi-\'etale cover.} We continue to work over a purely bosonic scheme.  Using the above notation, we say that the branched cover $\pi: \tils \to S$ determined by $(\pi_0: \tilc \to C, \tilde{L}, \pi_1^{\perp})$ is \emph{semi-\'etale} if $\pi_1^{\perp}:\pi_0^*L \lgr \tilde{L} $ is an isomorphism. If $\pi$ is semi-\'etale then $\ell=0$ in \eqref{localformramond}.
\newline

\begin{lemma} \label{semi}
 Let $S \to T$ be a family of supercurves over a purely bosonic scheme $T$, and fix a prime divisor $\mc{B} \subset S$. There is a one-to-one correspondence between branched covers of $C=S_{bos}$ with branch locus $B=\mc{B}_{bos}$ and semi-\'etale branched covers of $S$ with branch locus $\mc{B}$.  
 
\end{lemma}

\begin{proof}  
 Let $\pi_0: \tilc \to C$ be a branched cover of $C=S_{bos}$ and let $B \subset C$ and $R \subset \tilc$ denote its branch and ramification divisor. First, note that since $S$ is a supercurve over a bosonic scheme, $B$ extends uniquely to a divisor, $\mc{B}$, on $S$, by  Lemma \ref{evensame}. To prove the lemma, we will show that 
 the morphism $\pi: S \times_C \tilc \to S$ in the cartesian diagram  
\begin{equation}\label{di}
    \begin{tikzcd}
   S \times_{C} \tilc \arrow[r, "\pi"] \arrow[d] & S \arrow[d] \\
    \tilc \arrow[r, "\pi_0"] & C
\end{tikzcd} 
\end{equation} 
is the unique extension, up to isomorphism,  of $\pi_0$ to a semi-\'etale  cover of $S$ branched along $\mc{B}$. It suffices to show that $\pi$ is the branched cover of $S$ determined by the data $(\pi_0, (\pi_0)^*L,\on{id}_{\pi_0^*L})$. That this is the case, follows from 
\[ \mc{O}_{S \times_C \tilc}= \mc{O}_{\tilc} \oplus \Pi \pi_0^* L. \] 



\end{proof}

 We define the \emph{local degree} of $\pi$ at a connected component of the ramification divisor $\mc{P} \subset \mc{R}$ to be the local degree of $\pi_{bos}$ at $P \subset R$, i.e., the number of sheets of $\pi_{bos}$ that come together at the ramification point $P$.
These numbers are arranged into a sequence $\rho=(d_1, \dots, d_s)$ where $\sum d_i =d$, assuming the ramification points are ordered.
According to Theorem 4 in \cite{husemoller1962ramified} a degree $d$ branched cover $\pi: \tilc \to C$ of ordinary curves, with ramification configuration $\rho$, exists if and only if $g(C) \ge 2$ and the sum of the local degrees $d_i$ minus $1$ is congruent to $0$ modulo $2$, i.e., \[ \sum d_i - 1 \equiv 0 \ (\text{mod} \  2). \] 
From Lemma \ref{semi} we know that, at least in the split case,  every branched cover of $C=S_{bos}$ lifts uniquely to a branched cover of $S$. In particular, the cited existence result can also be used for supercurves.




\subsubsection{Branched covers of super Riemann surfaces}
Let $X=(S, \mc{D})$ be a super Riemann surface. We define a \emph{branched cover of $X$} to be a branched cover $\pi: \tils \to S$ of the supercurve $S$  such that $\pi^*(\mc{D})$ is a superconformal structure on $\tils$, possibly with Ramond punctures $\mc{P}_R$. 
An isomorphism of branched covers of super Riemann surfaces is an isomorphism of branched covers of supercurves with the additional condition that the isomorphisms respect  all superconformal structures. 
\newline

\noindent \emph{Local description.} Let $\pi: \tilx \to X$ be a branched cover,  and let $\mc{P} \subset \mc{R}$ be a connected component of the ramification divisor of $\pi$ of \emph{even} local degree.  Then $\mc{P}$ is a Ramond puncture on $\tilx$, and near $\mc{P}$, we can find (local) coordinates $(w, \theta')$ and $(z, \theta) $ on  $\tilx$ and $X$, respectively, such that $\mc{B} \subset X$  is defined by the equation $z=0$ and  $\pi$ is locally of the form
\begin{equation} \label{localdescRSR} z=w^{2\ell}, \ \ \theta=w^{\ell} \theta' \end{equation}
for some $\ell \ge 1$. If $\mc{P}$ is a connected component of odd local degree, then it is a NS puncture on $\tilx$, and $\pi$ is locally of the form 
\begin{equation} \label{localdescRSNS} z=w^{2\ell+1}, \ \ \theta=w^{\ell} \theta' \end{equation}
for some $\ell \ge 1$.
\newline

\noindent \textbf{Blow up construction.} Given a super Riemann surface $X =(S, \mc{D})$, and a branched cover $\pi: \tils \to S$, locally of the form $z=w^{\ell}, \theta=\theta'$, one may be tempted to think  $\pi: \tilx \to X$, where $\tilx=(\tils, \pull(\mc{D}))$, is a branched cover of $X$.  However, this is \emph{never} the case (unless $\pi$ is unramified).   The problem is that $\pull(\mc{D})$ will always have  parabolic degenerations along the ramification locus of $\pi$. 
In \cite{donagi2015supermoduli}, it is explained how these parabolic degenerations can be resolved in a way that recovers a superconformal structure on a certain supercurve canonically associated to $\tils$. The idea is to blow up $\tils$ along the distinguished closed points on each ramification divisor. The result is a supercurve $\on{Bl}(\tils)$ with the same bosonic reduction $\tilc$ as $\tils$, and with a natural projection $q:\on{Bl}(\tils) \to \tils$ such that $q^*(\pull \mc{D})$ is a superconformal structure on $\on{Bl}(\tils)$ with Ramond punctures along the exceptional divisors on $\on{Bl}(\tils)$ corresponding to the pullbacks by $q$ of 
ramification points of even local degree.

We recall that the blow up of an ordinary (smooth) curve at a point on that curve is just another copy of the curve. In particular, the bosonic reduction of $\on{Bl}(\tils)$ is the ordinary curve $\tilc$ and the composition $q \circ \pi$ reduces to $\pi_{bos}: \tilc \to C$.
\newline

 Lemma \ref{semi} has the following corollary:
\newline 

\begin{corollary} \label{bosmodriemann} Let $X=(S, \mc{D})$ be a family of super Riemann surfaces over a bosonic scheme $T$, and fix a prime divisor $\mc{B} \subset S$. Then there is a one-to-one correspondence between branched covers of $C=S_{bos}$ with branch locus $B=\mc{B}_{bos}$ and branched covers of $X$ with branch locus $\mc{B}$. 
    
\end{corollary}

\begin{proof} A  branched cover $\pi: \tilx=(\tils, \rpunc, \tilsup) \to S=(X, \mc{D})$  is   semi-\'etale because $(w,\theta':=\theta/w^l)$ on a  supercurve, and one can check in local coordinates that the generator $dz-\theta d\theta$ of $\mc{D}^{-2}$ on $X$ is mapped to $dw - w d \theta' d\theta'$ generating $\tilsup^{-2}(-\rpunc)$. Now apply Lemma \ref{semi}.  

\end{proof}

\section{Moduli of branched covers} \label{modbranch}

Let $\compbranch: \on{sSch} \to \on{Groupoid}$ denote the functor sending a superscheme $T$ to the groupoid \[ \compbranch(T)=  \begin{cases} \text{ \textbf{Objects}: Degree} \  d \ \text{branched covers} \  \pi: \tilx \to X  \ 
\text{of all} \   X \in \frak{M}_{2,1}(T), \\  \ \text{with specified ramification configuration} \  \rho  \  \text{and branched along the one NS puncture on} \  X.  \\ 
\text{ \textbf{Morphisms}: Isomorphisms of branched covers of super Riemann surfaces.} 
\end{cases} 
 \]  

The following examples of  moduli functors may be familiar:
\newline 

\noindent \textbf{Moduli of Curves.}  $M_{g,n}$ sends an ordinary scheme $T$ to the groupoid of families of genus $g$ curves with $n$ marked points over $T$.  
\newline 

\noindent \textbf{Moduli of Spin Curves.}  $SM_{g,n_S, n_R}$ sends an ordinary scheme $T$ to the groupoid of families genus $g$ spin curves with $n_S$ marked points (NS punctures) and $n_R$ Ramond punctures over $T$
\newline

\noindent \textbf{Supermoduli space.}  $\frak{M}_{g,n_S, n_R}$ sends a superscheme $T$ to the groupoid of families of genus $g$ super Riemann surfaces with $n_S$ marked points (NS punctures) and $n_R$ Ramond punctures over $T$. 
\newline 

The bosonic reduction of any functor $F: \on{sSch} \to \on{Groupoid}$ is its restriction to the subcategory $\on{Sch} \subset \on{sSch}$ of ordinary schemes, i.e., $F_{bos}: \on{Sch} \to \on{Groupoid}$. 
\newline

The bosonic reduction of $\frak{M}_{g,n_S, n_R}$ is $SM_{g,n_S,n_R}$. This follows immediately from the one-to-one correspondence between families of super Riemann surfaces over ordinary schemes and  families of spin curves over the same ordinary scheme. To work with the  bosonic reduction of $\compbranch$, 
we use the following cartesian diagram, which is the
 moduli version of Corollary \ref{bosmodriemann},
 \begin{equation} \label{finite}
    \begin{tikzcd} (\compbranch)_{bos} \arrow[r, "p_{bos}"] \arrow[d] & SM_{2,1} \arrow[d, ]\\
    \tilde{M}_{d,\rho} \arrow[r,] & M_{2,1}, 
    \end{tikzcd}
\end{equation}
where $SM_{2,1} \to M_{2,1}$ forgets the spin structure and makes $SM_{2,1}$ into a finite, \'etale cover of $M_{2,1}$,  $p_{bos}$ is the bosonic reduction of the morphism $p: \compbranch \to \modtwoone$ that sends $\pi: \tilx \to X$ to $X \in \modtwoone$ and $(\tilde{\phi}, \phi)$ to $\phi$, and $\tilde{M}_{d, \rho}$ is the moduli of branched covers of $M_{2,1}$ and $\tilde{M}_{d, \rho} \to M_{2,1}$ is the obvious projection.

The moduli functors $M_{g,n}$,  $SM_{g,n_S, n_R}$, $\frak{M}_{g,n_S, n_R}$, and $\compbranch$ each have a representable finite \'etale cover by a manifold, or supermanifold in the case $\frak{M}_{g,n_S, n_R} $ and $\compbranch$. The existence of such a cover means that these functors are Deligne-Mumford stacks, or superstacks in the case of $\frak{M}_{g,n_S. n_R}$ and $\compbranch$.  We prove that $\compbranch$ is Deligne-Mumford in Theorem \ref{app: DMstack} in the appendix of this paper. Henceforth, we will always confuse a Deligne-Mumford stack, or  superstack, with its \'etale cover, and all morphisms between Deligne-Mumford superstacks with the morphisms induced on their respective \'etale covers. 
\newline




The tangent set to a general functor $F: \on{sSch} \to \on{Groupid}$ at a point $f \in F(\Spec \bb)$ is the set of isomorphism classes in the groupoid $F_{/f}(\Spec D)$, i.e., isomorphism classes of first order deformations of $f$. 
where $F_{/f}$ is the groupoid of objects in $F$ which pullback to $f$ on $\Spec \bb$, and where $D$ denotes the super dual numbers. The set of isomorphism classes of first order deformations of objects in $\compbranch$, $\comp$, and $\modtwoone$, are described in Theorem \ref{app:deformation theorem},  Lemma \ref{deformRamond}, and Lemma \ref{definition}, respectively. 
\newline



\noindent The rest of this section deals with  proving the following two lemmas. 
\begin{lemma} \label{S1} The morphism $p: \compbranch \to \modtwoone$  makes $\compbranch$ into a finite cover of $\frak{M}_{2,1}$. In particular, $\compbranch$ is not projected by Corollary \ref{twoeight}.  \end{lemma}

We let  $i: \compbranch \to \comp \to \compplain$, where $\comp \to \compplain$ forgets the NS punctures,  denote the morphism sending $\pi: \tilx \to X$ to $\tilx \in \compplain$. Note we are using the same notation $\tilx$ for the super Riemann surface  on which the NS punctures have been forgotten. 

    \begin{lemma} \label{S2} The morphism $i: \compbranch \to \compplain$ is an immersion of supermanifolds. Furthermore, the normal sequence associated to the immersion $i_{bos}: (\compbranch)_{bos} \to SM_{g,0, 2r}$ is split exact, and thus $\compplain$ is not projected by Corollary \ref{twoeleven}. \end{lemma}
 

    

\noindent Assuming the proof of Lemmas \ref{S1} and \ref{S2}, the proof of Theorem \ref{main:intro} goes as follows: 

\begin{proof}[Proof of main result.] The possible values of $g$ and $2r$ for which there exist a branched covering $\tilx \in \compplain$ for some $X \in \modtwoone$ are described in Lemma \ref{range:intro}. For this range of values, the associated supermoduli spaces $\compplain$ are all not projected by Lemma \ref{S2}. 
    
\end{proof}

\begin{proof}[Proof of Lemma \emph{\ref{S1}}.] Recall that a morphisms of supermanifolds is finite if and only if its bosonic reduction is a finite morphism of manifolds.  The bosonic reduction of $p$, denoted as $p_{bos}$, is the top horizontal arrow in diagram \eqref{finite}. It suffices to show that $p_{bos}$ is finite.  From diagram \ref{finite} we can see that $p_{bos}$ is the base change of the finite map $\tilde{M}_{d,\rho} \to M_{2,1}$, and is therefore itself finite. 
\end{proof}

\begin{proof}[Proof of Lemma \emph{\ref{S2}}:] 
The morphism 
\[ dp \vert_{\pi}: (T \compbranch) \vert_{\pi} \to H^1(X, \mc{A}_X) \] 
is an isomorphism by Theorem \ref{app:deform}. Given this isomorphism , the differential of $i$ at $\pi$ is 
\[ di \vert_{\pi}: H^1(X, \mc{A}_X) \to H^1(\tilx, \mc{A}_{\tilx}). \] 

The map $di$ is injective if and only if the pre-image of the trivial deformation $\tilx \times \Spec D$ in $H^1(\tilx, \mc{A}_{\tilx})$ is isomorphic to the trivial deformation of $X \times \Spec D$.  Suppose $di(\pi') = \tilx \times \Spec D$, then again by Corollary \ref{twoeight} we must have  $X' \cong X \times \Spec D $. 

For the second statement of the lemma, we need to show that the normal sequence associated to $i_{bos}$,
\[ 0 \to T (\compbranch)_{bos} \to i_{bos}^* T SM_{g,0,2r} \to N \to 0 \]
is split exact. To show this, we can 
repeat the proof of
Proposition 5.3 in \cite{donagi2015supermoduli}.
 
\end{proof}

\section{Appendix} 

In this section, we prove (Theorem \ref{app: DMstack}) that $\compbranch$ is a Deligne-Mumford superstack. We will first show that $\compbranch$ satisfies the conditions of the super Artin theorem, Theorem \ref{app: artin}, and is, therefore, an algebraic superstack. Once we have established that $\compbranch$ is algebraic, we can apply  Theorem 8.3.3 in \cite{olsson2016algebraic}: An algebraic stack $\mcx$ is Deligne-Mumford if and only if the diagonal map $\Delta: \mcx \to \mcx \times \mcx$ is formally unramified. The condition that $\Delta$ be formally unramified is equivalent to the condition that the automorphism group $\on{Aut}_{x}$ for $x \in \mcx(\Spec \bb)$ is a reduced finite group scheme over $\Spec \bb$. We will show this holds for $\pi \in \compbranch(\Spec \bb)$, concluding that $\compbranch$ is a Deligne-Mumford superstack. 

\subsection{Preliminaries}

In this section, we define the various functors, e.g., $D_v(M)$, $\on{Ob}_v(M)$, that appear in the statement of the super Artin theorem.
\newline 

\noindent \emph{Extensions of objects.} Let $\mc{X}:\sS^{op} \to \on{Groupoid}$ be a superstack over the category of affine superschemes, and let $A_0$ be a superalgebra. An \emph{extension} of an object $v \in \mc{X}(\Spec A_0)$ is a pair $(f:B \to A_0, v')$ where $f$ is a morphism in $\sS$ and $v'$ is an object in $\mc{X}(\Spec B)$ such that $v' \times_{\Spec B} \Spec A_0 \cong v$. 
We will define two categories, denoted by $\on{Def}_{v}(M)$ and $D_{v}(M)$, that keep track of certain classes of extensions of $v$. 
\newline

\noindent \emph{Extensions of superalgebras.} An \emph{(infinitesimal) extension} of a superalgebra $A_0$ by a coherent $A_0$-module $M$ is a superalgebra $A$ fitting into the short exact sequence
\[0 \to M \to A \to A_0 \to 0. \]
We will henceforth write extension when we mean infinitesimal extensions. 
Note that $M$ inherits a multiplication $M \cdot M$ from $A$, and the $A_0$-module structure on $M$ forces $M \cdot M=0$. We say that an extension $A$ is \emph{trivial} if $A \cong A_0 \oplus M$. An important example of a trivial extension is the super dual numbers, $\bb[t, \eta]/(t^2, \eta t),\ |t|=0, |\eta|=1$ in $\on{Ext}(\bb, \bb \oplus \Pi \bb)$.

We let $\on{Ext}(A_0, M)$ denote the set of isomorphism classes of extensions of $A_0$ by $M$, and let $\on{Coh}(A_0)$ denote the category of all coherent modules over $A_0$.  $\on{Ext}(A_0, M)$ is functorial in $M$; that is,  if $\phi: M \to M'$ is a morphism of $A_0$-modules, then there exists morphism $\on{Ext}(A_0,M) \to \on{Ext}(A_0,M')$ sending an extension $A$ to its pushout $\phi_* A$ in $\on{Ext}(A_0,M')$.  In other words, $\on{Ext}(A_0,-)$ is a category fibered in Sets over $\on{Coh}(A_0)$.  
\newline

\noindent \emph{Extension functors.} Fix $v \in \mc{X}(\Spec A_0)$, and $M \in \on{Coh}(A_0)$. The following functor keeps track of the extensions of $v$ to the Spec of all extensions of $A_0$ by all coherent $A_0$-modules:   
\begin{align*}
    \on{Def}_v(M) &: \on{Ext}(A_0,M) \to \on{Set} \\ 
    {} & f:A \to A_0 \mapsto  \on{Def}_v(M)(A)=\{ \text{isomorphism classes of extensions of $v$ over } \Spec A \}  
\end{align*}  
$\on{Def}_v(M)$ is functorial in $M$; that is, if $\phi: M \to M'$ is a morphism of $A_0$-modules, then there exist a morphism $\on{Dev}_v(M) \to \on{Def}_v(M')$ sending an extension $v' \in \on{Def}_v(M)(A)$ to its pushout $\phi_* v' \in \on{Def}_v(M')(\phi_*A)$.
We let
$\mc{D}ef_{v}(\mm)$ denote the sheafification of the presheaf 
$\Spec A_0 \supset U \mapsto \on{Def}_{v \vert_U}(\mm \vert_U)$ on $\Spec A_0$, where $\mm$ is the coherent sheaf of $\mc{O}_{\Spec A_0}$-modules associated to $M$.

The following functor keeps track of all isomorphism classes of extensions of $v$ to the Spec of all \emph{trivial extensions} of $A_0$ by all coherent $A_0$-modules: 
\begin{align*}
    D_{v} &: \on{Coh}(A_0) \to \on{Set} \\
    {} & M \mapsto D_{v}(M) = \{ \text{isomorphism classes of extensions of} \ v \ \text{to} \ \Spec(A_0 \oplus M) \}.
\end{align*}
where $A_0 \oplus M$ is the trivial extension of $A_0$ by $M$.  
We we write $\mc{D}_{v}(\mm)$ for the sheafification of the presheaf $\Spec A_0 \supset U \mapsto D_{v \vert_U}(\mm \vert_U)$ on $\Spec A_0$.
\newline

\noindent \emph{Obstructions. } There is a natural morphism, functorial in $M$, $\phi: \on{Def}_{v}(M) \to \on{Ext}(A_0,M)$  sending an extension $(f:A \to A_0, v')$ to $(f:A \to A_0) \in \on{Ext}(A_0,M)$. The extensions of $A_0$ by $M$ for which there exists an obstruction to lifting $v$ are parameterized by, 
\[ \on{Ob}_v(M):=  \on{Ext}(A_0,M)/\phi(\on{Def}_{v}(M)).\]
This definition was introduced in \cite{flenner1981kriterium}. Everything we have said in this section holds equally well for the sheaves $\mc{E}xt(A_0,\mm)$, $\mc{D}ef_{v}(\mm)$ and $\mc{O}b_v(\mm)$.

\begin{theorem}[Super Artin Theorem, \cite{ott2021artin}] \label{app: artin}  Let $\mxx$ be a limit preserving superstack and assume that Schlessinger's conditions hold. (These conditions guarantee that the tangent set to $\mcx$ is a finite dimensional super vector space). Then $\mxx$ is an algebraic superstack locally of finite type if and only if
\begin{enumerate}
    \item[\emph{(A1)}.] The diagonal morphism $\Delta: \mxx \to \mxx \times \mxx$ is represented by an algebraic superspace locally of finite type. 
    
    \item[\emph{(A2)}.] \emph{Effectivity}: If $R$ is a complete, local, Noetherian superalgebra, then the natural map 
    \[  \mxx( \spe R) \to \varprojlim \mxx( \spe R/\frk{R}^n)\] 
   is an equivalence of categories.

    \item[\emph{(A3)}.] \emph{Coherence of $\mc{O}b$}: For every superscheme $T$, object $v \in \mxx(T)$ and coherent $\otr$-module $\mc{M}$, $\mc{O}b_v(\mc{M})$ is a coherent $\ot$-module.
    
    
    \item[\emph{(A4)}.] \emph{Constructibility}: Let $v, T$ be as above and suppose that $T' \subseteq T$ is an irreducible reduced subspace of $T$. Then there is a Zariski open dense subset $V'\subseteq T'$ such that for each closed point $t \in V'$, the canonical morphisms 
\begin{itemize}
    \item[\emph{(a)}] $\mc{D}_v(\otpr) \otimes k(t) \to \mc{D}_v(\otpr \otimes k(t))_t$  and $\mc{D}_v(\Pi \otpr) \otimes k(t) \to \mc{D}_v(\Pi \otpr \otimes k(t))_t$ are  bijective, and
    \item[\emph{(b)}] $\mc{O}b_v(\otpr) \otimes k(t) \to \mc{O}b_v(\otpr \otimes k(t))_t$ and $\mc{O}b_v(\Pi \otpr) \otimes k(t) \to \mc{O}b_v(\Pi \otpr \otimes k(t))_t$  are injective. 
\end{itemize}

\item[\emph{(A5)}.] The sheaf $\mc{D}$ is compatible with completion: If $T_0$ is a reduced finite type scheme, $a_0 \in \mxx(T_0)$, $t \in T_0$ and $\mm$ is a coherent $\oto$-module, then \[ \mc{D}_{a_0}(\mm) \otimes_{\mc{O}_{T_0}} \wh{\mc{O}}_{T_0,t} \lgr \varprojlim \mc{D}_{a_0}(\mm/\frak{m}^n \mm)_t \]
    
    \item[\emph{(A6)}.]  The sheaf $\mc{D}$ is compatible with \'etale localization: If $e: S_0 \to T_0$ is an \'etale morphism of reduced schemes, with $e^*(a_0)=b_0$, and $\mm$ is a coherent $\oto$-module, then
\[
    \mc{D}_{b_0}(\mm \otimes_{\oto} \mc{O}_{S_0}) \cong \mc{D}_{a_0}(\mm) \otimes_{\oto} \mc{O}_{S_0}
\]
 \end{enumerate}
\end{theorem}

\subsection{Deformation theory of branched covers.}  In this section, we will describe the deformation theory of the moduli superstack $\compbranch$ of branched covers of super Riemann surfaces.

 By the deformation theory of a (reasonable) superstack $\mcx$, we mean the following three objects associated with a fixed, but arbitrary, object $v \in \mcx(\Spec \bb)$: The group of infinitesimal automorphisms of $v \times \Spec D$, the set of isomorphism classes of first order deformations of $v$, and the obstructions to lifting $v$. The group of infinitesimal automorphisms is the tangent space of automorphism group $\on{Aut}(v)$ at $\on{id}_v$, and the set of isomorphism classes of first order deformations of $v$ is the tangent set to $\mcx$ at $v$. We do not know a tangent description for the obstructions to lifting.

\begin{example} \normalfont Let $X \in \comp(\Spec \bb)$. The group infinitesimal automorphisms $\on{Aut}_{/X}(X \times \Spec D)$ of $X'$ is equal to $H^0(X, \mc{A}_X)$, the set of isomorphism classes of first order deformations of $X$ is equal to $H^1(X, \mc{A}_X)$, and the obstructions to lifting $X$ are parameterized by $H^2(X, \mc{A}_X)$. 
    
\end{example}

\noindent \emph{First order deformations.} Let $D$ denote the super dual numbers , $\bb[t, \eta]/(t^2, \eta t), \ |t|=0, |\eta|=1$, and let $(\pi: \tilx \to X) \in \compbranch(\Spec \bb)$. The tangent set to $\compbranch$ at $\pi$ is the set $\on{Def}_{\pi}:=\on{Def}_{\pi}(\Spec D)$ of  isomorphism classes of first order deformations of $\pi$, that is triples $(\tilx', X', \pi')$ where $\tilx',X'$ are first order deformations of $\tilx, X$, respectively, and $\pi': \tilx' \to X'$ is a morphism making the following diagram commute  
\begin{equation} \label{app:deform}
    \begin{tikzcd} 
        \tilx \arrow[r] \arrow[d, "\pi"] & \tilx' \arrow[d,"\pi'"] \\
        X \arrow[r] & X'. 
    \end{tikzcd}
\end{equation}
 We say that $(\tilx', X',\pi')$ is \emph{trivial} if there exists an isomorphism of branched covers \[ (\tilx', X', \pi') \lgr (\tilx \times \Spec D, X \times \Spec D, \pi \times \on{id}_D). \]  
 \newline 
 
\noindent \emph{Infinitesimal automorphisms.} The restriction of an automorphism of a branched cover $\phi: \pi \times  \Spec D \lgr \pi \times \Spec D$ to $\Spec \bb$, denoted by $\phi \vert_{\bb}$, is an automorphism of $\pi: \tilx \to X$, and we call
$\phi$ an \emph{infinitesimal automorphism} if $\phi \vert_{\bb} \in \on{Aut}(\pi)$ is equal to the identity morphism $ \on{id}_{\pi} \in \on{Aut}(\pi)$. The group of infinitesimal automorphisms is, denoted by $\on{InfAut}(\pi)$, is the tangent space of $\on{Aut}(\pi)$ at $\on{id}_{\pi} \in \on{Aut}(\pi)$. 
\begin{lemma} $\on{InfAut}(\pi)=\{ \on{id} \}$. In particular, there are no non-trivial infinitesimal automorphisms.
\end{lemma}
\begin{proof} $\on{Aut}(\pi) \subset \on{Aut}(X) \times \on{Aut}(\tilx)$ and $\on{Aut}(X),\on{Aut}(\tilx)$ are finite for $g >1$. 
\end{proof}

\noindent \emph{Obstructions.} There are no obstructions to lifting $\pi$ by Proposition 3.3 in \cite{donagi2015supermoduli}, appropriately adapted to the case that $\pi$ has a non-zero number of ramification points of even local degree. 
\newline

The main result of this section is the following theorem which identifies the set of isomorphism classes of first order deformation of a branched cover $\pi: \tilx \to X$ with the super vector space  $H^1(X, \mc{A}_X)$. 
\begin{theorem} \label{app:deformation theorem} $\on{Def}_{\pi}= H^1(X, \mc{A}_X)$. In particular, the tangent set to $\compbranch$ at $\pi$ naturally admits the structure of a super vector space of dimension $4|3$. 
    
\end{theorem}

The proof of Theorem \ref{app:deformation theorem} occupies the rest of this section. 
\newline 

\noindent \emph{Variants on $\on{Def}_{\pi}$.} Let $\on{Def}_{\pi}(X')$ denote the set of all isomorphism classes of first order deformations of $\pi$ with $X'$ fixed. The elements of $\on{Def}_{\pi}(X')$ are pairs $(\tilx', \pi')$ where $\pi': \tilx' \to X'$ is a  morphism making diagram \eqref{app:deform} commute. Similarly, let $\on{Def}_{\pi}(\tilx',X')$ denote the set of all isomorphism classes first order deformations of $\pi$ with $\tilx'$ and $X'$ fixed. The elements of $\on{Def}_{\pi}(\tilx', X')$ are morphisms $\pi': \tilx' \to X'$ making the diagram \eqref{app:deform} commute.   
\newline

Let $(\pi: \tilx \to X) \in \compbranch(\Spec \bb)$, let $\nspunc$ denote the one NS puncture on $X$, and recall that $\nspunc$ is equal to the branch locus of $\pi$. 
Set $F:=\pi^{-1}(\nspunc) \subset \tilx$, and recall that the ramification divisor $\mc{R} \subset F$, and recall that $\mc{R}=\tilde{\mc{P}}_R + \tilde{\mc{P}}_{NS}$, where $\tilde{\mc{P}}_R$ is the divisor of ramification divisors of even  local degree and corresponding to Ramond punctures, while $\tilde{\mc{P}}_{NS}$ is the divisor of ramification divisors (or, by duality, ramification points), of odd local degree and corresponding to NS punctures. The divisors (or, by duality, points) in the complement $F -  \mc{R}$ also correspond to NS punctures. In fact, in the examples we use for proving the main theorem, $\nspunc$ is empty because none of the local degrees is $>2$, so the only NS points we get are in $F-\mc{R}$.

\begin{lemma} \label{appi:hzero} 
 There is a one-to-one correspondence between $\on{Def}_{\pi}(\tilx', X')$ and $H^0(\tilx, \mc{A}_{\tilx}(-F)$, where $\pi^{-1}(\nspunc)$ is the pre-image of the one NS puncture $\nspunc$ on $X$.  In particular, since \[ H^0(\tilx, \mc{A}_{\tilx}(-F))=0, \]
 there is at most one map $\pi': \tilx' \to X'$ making the diagram \eqref{app:deform} commute for any fixed pair of first order deformations $\tilx', X'$ of $\tilx, X$, respectively. 
    
\end{lemma}

\begin{proof} Let $\psi$ denote the composition $\tilx \overset{\pi}{\to} X \to X'$ in diagram \eqref{app:deform} and set $I=\on{ker}(\mc{O}_{\tilx'} \to \mc{O}_{\tilx})$.  By flatness, $I \cong f_0^*(J)$ where $J=\on{ker}(D \to \bb)$ and $f_0: \tilx \to \Spec \bb$ is the structure morphism. In particular, $I$ is generated by $t$ and $\eta$ as a $\mc{O}_{\tilx}$-module, and therefore $I \cong \mc{O}_{\tilx} \oplus \Pi \mc{O}_{\tilx}$. 

There is a one-to-one correspondence between the set of arrows $\pi'$ filling in the dotted arrow in the below diagram, i.e., the set of all lifts of $\psi$, and the set $\on{Der}_{\pi}(\tilx', X')$. : 
    \begin{equation}
        \begin{tikzcd}
            \tilx \arrow[r] \arrow[d, "\psi"] & \tilx' \arrow[dl, dotted, "\pi'"] \\
            X' & {} 
        \end{tikzcd}
    \end{equation}To prove the lemma it therefore suffices to  show that there is bijection from the set of all lifts of $\psi$ to $H^0(\tilx, \mc{A}_{\tilx}(-\mc{P}_R))$

    Since $X'$ is smooth over $\Spec D$, there always exists a lift $\pi'$ of $\psi$ locally on $\tilx$. The difference between any two such lifts is an even superconformal derivation $d \in \on{Der}_D^{scf}(\mc{O}_{X'},\ \psi_*I)(-\nspunc')$ with values in $\psi_*I$, and preserving the one NS puncture $\nspunc'$ on $X'$, i.e., the a lift of $\psi$ is a section of $H^0(X',\on{Der}_D^{scf}(\mc{O}_{X'},\ \psi_*I)(-\nspunc'))$.  The proof of the lemma now follows from the following isomorphisms: 
    \begin{align*} \on{Der}_D^{scf}(\mc{O}_{X'},\ \psi_*I)(-\nspunc') & = \on{Der}_D^{scf}(\oxprime, \psi_*I \otimes \mc{O}(-\nspunc')) & (\text{see} \ \eqref{NSder}) \\
    {} & = \on{Der}_D^{scf}(\oxprime(\nspunc'), \psi_* I) & {} \\
    {} & = \on{Der}_k^{scf} (\psi^*\oxprime(\nspunc'), I) & (\text{Adjunction}) \\
     {} &= \underline{\on{Der}}_k^{scf}(\psi^* \oxprime(\nspunc'), \mc{O}_{\tilx}) & (\text{Lemma} \ \ref{technical}) \\ 
     {} & = \underline{\on{Der}}_k^{scf}(\mc{O}_{\tilx}(F), \mc{O}_{\tilx}) & \psi^*\nspunc'=F, \ \psi^* \oxprime \cong \mc{O}_{\tilx} \\
    {} & =  \mc{A}_{\tilx}(-F) & (\text{see} \ \eqref{longlist}).
    \end{align*}

\end{proof}

We can use the morphisms $D \to k[t]/t^2$ and $D \to k[\eta]$, to pullback $\pi$ to an even and odd first order deformation of $\pi$,  respectively. This breaks $\on{Def}_{\pi}(\Spec D)$ into a direct sum of its even component $\on{Def}_{\pi}^+=\on{Def}_{\pi}(\Spec k[t]/t^2)$ and odd component $\on{Def}_{\pi}^-=\on{Def}_{\pi}(\Spec k[\eta])$ corresponding to the even and odd components of the tangent space to $\compbranch$ at $\pi$.

    

\begin{lemma} \label{appi:oddtangentspace} If $(\tilx', \pi') \in \on{Def}_{\pi}^-(X \times \Spec k[\eta])$, then $\tilx' \cong \tilx \times \Spec k[\eta]$ and $\pi' \cong \pi \times \Spec k[\eta]$.
\end{lemma}

\begin{proof}
Suppose $\pi': \tilx' \to X \times \Spec [\eta]$ is an odd first order deformation of $\pi: \tilx \to X$, and note that $X \times \Spec[\eta]$ is split since it is a \emph{trivial} first order odd deformation of $X$.  Furthermore, since $\pi'$ is a finite covering of a split super Riemann surface, $\tilx'$ must be split by Corollary \ref{twoeight}, and thus $\tilx \cong X' \times \Spec D$, and $\pi' \cong \pi \times \on{id}$ by Lemma \ref{appi:hzero}.
    
\end{proof}

\begin{proof}[Proof of \emph{Theorem} \ref{app:deformation theorem}] We will prove that differential at $\pi$ of the projection $p: \compbranch \to \modtwoone$ is an isomorphism, i.e., that \[ dp \vert_{\pi}: T\compbranch \vert_{\pi} \to T \modtwoone \vert_{p(\pi}\]
is an isomorphism of super vector spaces. The morphism $p$ is surjective by definition, and so  we are left to show that $dp \vert_{\pi}$ is injective. 

The tangent space to $\compbranch$ at $\pi$ is the set $\on{Def}_{\pi}$ of isomorphism classes of first order deformations of $\pi$. The tangent space to $\modtwoone$ at $X=p(\pi)$ is the set of isomorphism classes of first order deformations of $X$ and we know from Lemma \ref{deformSRSns} that the set of isomorphism classes of first order deformations of $X$ 
is canonically isomorphic to $H^1(X, \mc{A}_X(-\nspunc))$, where $\nspunc$ denotes the one NS puncture on $X$. Under these identifications,  $\dpi: \on{Def}_{\pi} \to H^1(X, \mc{A}_X(-\nspunc))$ sends a isomorphism class of first order deformations $\pi': \tilx' \to X'$ to the isomorphism class of the first order deformation $X'$ of $X$.

The identity object $0$ in $H^1(X, \mc{A}_X(-\nspunc))$ is the class of the trivial deformation, $X \times \Spec D$, of $X$, and the identity object $0$ in $\on{Def}_{\pi}$ is the class of the trivial deformation $\pi \times \on{id}_D: \tilx \times \Spec D \to X \times \Spec D $ of $\pi$. 
To prove that $\dpi$ is injective we will show that $\dpi^{-1}(0)=0$.

The map $\dpi$ is a direct sum of its even and odd components, 
\[ \dpi^+: \on{Def}_{\pi}^+ \to H^1(X, \mc{A}_X(-\nspunc))^+, \ \ \dpi^-: \on{Def}_{\pi}^- \to H^1(X, \mc{A}_X(-\nspunc))^-. \]
Set $D^+=\Spec \bb[t]/t^2$ and $D^-= \Spec \bb[\eta]$.
The identity objects in $\on{Def}_{\pi}^+$ and $\defpi^-$ are $\pi \times \on{id}_{D^+}$ and $ \pi \times \on{id}_{D^-}$, respectively. The identity objects in $H^1(\mc{A}_X(-\nspunc))^+$ and $ H^0(\mc{A}_X(-\nspunc))^-$ are $X \times D^+$ and $X \times D^-$, respectively. 
The map $\dpi^-$ is injective by Lemma \ref{appi:oddtangentspace}: the lemma says that any first order deformation of $\pi$ with target the trivial deformation of $X$, i.e., $\pi': \tilx' \to X \times D^+$,  is isomorphic to the trivial deformation. 
For $\dpi^+$: Since $D^+$ is purely bosonic, $ \compbranch(D^+)=(\compbranch)_{bos}(D^+)$, and \[ \dpi^+=dp_{bos} \vert_{\pi_{bos}}: T(\compbranch)_{bos} \vert_{\pi_{bos}} \to T SM_{2,1} \vert_{p_{bos}(\pi_{bos})}.\] The diagram \ref{finite} shows $p_{bos}$ as the pullback of the \'etale morphism $\tilde{M}_{2,1} \to M_{2,1}$. Since \'etale morphisms are stable under pullback,   $p_{bos}$ is \'etale, and  $\dpi^+=d p_{bos} \vert_{\pi_{bos}}$ is therefore an isomorphism. 
\end{proof}

\subsection{$\compbranch$ is a Deligne-Mumford superstack}
A superstack $\mc{X}$ is said to be \emph{Deligne-Mumford} if there exist a superscheme $U$ and a morphism $U \to \mc{X}$ which is representable \footnote{ A morphism of superstacks $\mc{Y} \to \mc{X}$ is said to be \emph{representable} if for every superscheme $T$ and morphism $T \to \mc{X}$ the fiber product $\mc{Y} \times_{\mc{X}} T$ is representable by a superscheme. Let $P$ be a property of morphisms of superschemes which is stable under pullback, e.g., smooth, \'etale, proper. A representable morphism $\mc{Y} \to \mc{X}$ has property $P$ if $\mc{Y} \times_{\mc{X}} T \to T$ has property $P$. }, surjective, and \'etale. 

\begin{theorem} \label{app: DMstack} $\compbranch$ is a smooth, proper Deligne-Mumford superstack. 
    
\end{theorem}

\begin{proof} We first prove that $\compbranch$ is an algebraic superstack by showing that $\compbranch$ satisfies the conditions of the super Artin theorem. That $\compbranch$ is Deligne-Mumford then follows from the fact that the automorphism group $\on{Aut}_{\pi}$ is finite, see Theorem \cite{olsson2016algebraic} in \cite{olsson2016algebraic}.

 \noindent (A1): We will prove that the diagonal morphism $\Delta: \compbranch \to \compbranch \times \compbranch$ is representable by an algebraic superspace locally of finite type. This condition is equivalent to the condition that for every pair $(\pi_1: \tilx \to X), (\pi_2: \tilx' \to X') \in \compbranch(T)$ the functor 
    \begin{align*} I:=& \on{Isom}(\pi_1, \pi_2): \on{sSch}/T \to \on{Set} \\
    {} & I(f:S \to T)= \{ \text{isomorphisms } \ f^*\pi_1 \lgr f^*\pi_2  \ \text{over} \ T \}\end{align*}
    is representable by a superscheme. 

    We will first prove that $I$ is representable in the case $\pi_1=\pi_2$, so that $I=\on{Aut}_T(\pi_1)$. Let $\on{Cov}_T(\tilx, X)$ denote the groupoid of all branched covers of $X$ by $\tilx$, and note that $\pi_1$ is an object in $\on{Cov}_T(\tilx, X)$. There is a natural action of $\on{Aut}_T(\tilx) \times \on{Aut}_T(X)$ on the set of objects in $\on{Cov}_T(\tilx, X)$ by conjugation,  \[ (\tilde{\phi}, \phi) \cdot \pi = \phi \circ \pi \circ \tilde{\phi}^{-1}. \]
    The automorphism group $\on{Aut}_T(\pi_1)$ is the stabilizer subgroup $\on{Stab}(\pi_1) \subset \on{Aut}_T(\tilx) \times \on{Aut}_T(X)$, and is therefore representable by a superscheme. Coming back to the general case,  the sheaf $\on{Isom}(\pi_1, \pi_2)$ is a $\on{Aut}_T(\pi_1)$-torsor, and is therefore representable.  
\newline 

That (S1'), (S2), (A3), (A4), (A5), (A6) hold follows from the identification of $D_{\pi}(M)$ with $H^1(X, \mc{A}_{X/A_0} \otimes_{\mc{O}_{X}} f_0^*(M))$ where $f_0: X \to \Spec A_0$ is the structure map, which is clearly a finite $A_0$-module, and $\on{Ob}_{\pi}(M)=0$ by Proposition 3.3 in \cite{donagi2015supermoduli}, adapted appropriately to the case where $\pi$ has a non-zero number of ramification points with even local degree. 
\newline

Condition (A2): if $\{\pi_n: \tilx_n \to X_n \}$ is a compatible sequence in $\compbranch(\widehat{A}_0/\frak{m}^{n})$, where $X_n,\tilx_n$ are super Riemann surface over $ \Spec A_0/\frak{m}^n$, then it can be approximated by some $(\pi: \tilx \to X) \in \compbranch(\widehat{A_0})$.   Since $\modtwoone$ and $\compplain$ are algebraic superstacks, there exists $X \in \modtwoone(\Spec \widehat{A}_0)$ which maps to $\varprojlim_n X_n$ under the equivalence $\modtwoone( \spe \widehat{A}_0) = \varprojlim \modtwoone( \spe \widehat{A}_0/\frak{m}^n)$, and similarly there exist $\tilx \in \compplain(\widehat{A}_0)$ approximating $\varprojlim_n \tilx_n$.  Now apply Corollary 8.4.6 \cite{FGA} which states that there is an equivalence from the category  of finite $X$-superschemes, proper over $\Spec(\widehat{A}_0)$ to the category of finite formal $\widehat{X}$-superschemes proper over $\varprojlim_n \Spec \widehat{A}_0/\frak{m}^n$. The cited corollary is a corollary of the Grothendieck existence theorem, a super version of which was proved in \cite{moosavian2019existence}.

That $\compbranch$ is smooth follows from the vanishing of $\on{Ob}$, and it is proper because $(\compbranch)_{bos}$ is proper. 
\end{proof}

 \printbibliography

\end{document}